\documentclass[a4paper, pdftex]{article}

\usepackage[utf8]{inputenc}
\usepackage[T1]{fontenc}
\usepackage[pdftex]{graphicx}
\usepackage{amsmath}
\usepackage{amsfonts}
\usepackage{amsthm}
\usepackage{amssymb}
\usepackage{setspace}
\usepackage[format=hang]{caption}
\usepackage{thmtools}
\usepackage{thm-restate}
\usepackage{mathtools}

\begin{document}
\newtheorem{thm}{Theorem}[section]
\newtheorem{cor}[thm]{Corollary}
\newtheorem{lem}[thm]{Lemma}
\newtheorem{fact}[thm]{Fact}
\newtheorem{prop}[thm]{Proposition}
\newtheorem{thmletter}{Theorem}
\renewcommand*{\thethmletter}{\Alph{thmletter}}
\theoremstyle{definition}
\newtheorem{defi}[thm]{Definition}
\theoremstyle{remark}
\newtheorem{rem}[thm]{Remark}
\newtheorem{exm}[thm]{Example}

\newcommand{\set}[1][ ]{\ensuremath{ \lbrace #1 \rbrace}}
\newcommand{\ra}{\ensuremath{\rightarrow}}
\newcommand{\bsl}{\ensuremath{\backslash}}
\newcommand{\grep}[2]{\ensuremath{\left\langle #1 | #2\right\rangle}}
\renewcommand{\ll}{\left\langle}
\newcommand{\rr}{\right\rangle}

\title{Maximal hyperbolic towers and weight in the theory of free groups}
  \author{Benjamin Br\"uck}
  \maketitle

  \maketitle

\begin{abstract}
We show that in general for a given group the structure of a maximal hyperbolic tower over a free group is not canonical:
We construct examples of groups having hyperbolic tower structures over free subgroups which have arbitrarily large ratios between their ranks. These groups have the same first order theory as non-abelian free groups and we use them to study
the weight of types in this theory. 
\end{abstract}

\section{Introduction}
Around 1945, Tarski asked the question whether all non-abelian free groups share the same first order theory. The affirmative was given independently by Kharlampovich and Myasnikov (\cite{MyasnikovElementary}) and Sela ({\cite{SelaDiophantine}}). However, being a free group is not a first order property. This means that in addition to the free groups, there are also non-standard models of their theory $T_{fg}$ (also called \emph{elementary free groups}), i.e. groups that share the same theory as free groups but are not free themselves. Sela gave a geometric description of all finitely generated models of $T_{fg}$ by introducing the notion of a hyperbolic tower. He showed that the following is true (see \cite[Theorem 6]{SelaDiophantine} and the comments on it in \cite{LPSTowers}):

\begin{fact}
\label{introduction trivial tower implies T_fg}
Let $G$ be a finitely generated group. Then $G$ is a model of $T_{fg}$ if and only if $G$ is non-abelian and admits a hyperbolic tower structure over the trivial subgroup.
\end{fact}

Furthermore and more surprisingly, Sela showed in \cite{SelaCommonTheory} that the common theory $T_{fg}$ of non-abelian free groups is stable. This provided a new and rich example of a group that is, on the one hand, a classical and complex structure but, on the other hand, tame enough in the model theoretic sense to allow the application of the various tools developed in stability theory. Conversely, the study of free groups in algebra and topology has brought forth many geometric methods that can now be used to refine stability-theoretic analysis.

This is the context in which this article is set. Motivated by model theoretic ideas, we seek to gain a better understanding of hyperbolic towers by applying geometric tools that include Whitehead graphs, Bass-Serre theory, and covering spaces.

If $G$ is a non-abelian, finitely generated group, we call a free subgroup $H\leq G$ a \emph{maximal free ground floor}, if $G$ admits a hyperbolic tower structure over $H$, but not over any other free subgroup in which $H$ is a free factor. From the perspective of model theory, a basis of $H$ now plays a similar role for the group $G$ that a basis plays for an arbitrary free group, meaning that both such sets have the same type and are maximal independent with respect to forking independence over $\emptyset$. This is clear if $G$ is a free group itself, but much more interesting if $G$ is a non-standard model of $T_{fg}$ where we have a priori no notion of a basis.
Our main result in the first part of this article is:

\begin{thmletter}
\label{Theorem differences in sizes}
For each $n\in \mathbb{N}$, there is a finitely generated group that has one hyperbolic tower structure over a maximal free ground floor of basis length $2$ and another tower structure over a maximal free ground floor of basis length $n+2$.
\end{thmletter}

We explicitely construct these different tower structures, building on ideas of Louder, Perin and Sklinos (see \cite{LPSTowers}).

Closely related to this is the weight of the type $p_0$ of a primitive element in a free group. Pillay showed in \cite{PilForking} that $p_0$ is the unique generic type over the empty set in $T_{fg}$. In general, if a type $p$ has finite weight, its weight bounds the ratio of the sizes of maximal independent sets of realisations of $p$. Hence, Theorem \ref{Theorem differences in sizes} can also be seen as an alternative proof for the infinite weight of $p_0$, a fact already proven by Pillay (\cite{PilGenweight}) and Sklinos (\cite{Rizinfweight}). 

In the last section of this article, we extend Sklinos' techniques in order to generalise this result as follows:

\begin{restatable}{thmletter}{Theoreminfiniteweight}
\label{Theoreminfiniteweight}
In $T_{fg}$, every non-algebraic (1-)type over the empty set that is realised in a free group has infinite weight.
\end{restatable}

The organisation of the article is as follows: We start in Section \ref{Section Bass Serre and towers} with a short account of Bass-Serre theory and surface groups before presenting the definition of a hyperbolic tower. In Section \ref{Section model theory}, we give some model-theoretic basics. Afterwards, we present a criterion for a subgroup to be a maximal free ground floor in Section \ref{Section maximal ground floors} and use this to prove Theorem \ref{Theorem differences in sizes} in Section \ref{Section build the towers}. Finally, Section \ref{Section infinite weight} contains more details about weight and introduces Whitehead graphs in order to prove Theorem \ref{Theoreminfiniteweight}.

The results of this article are taken from the author's master thesis. Many thanks are due to Rizos Sklinos and Tuna Altinel for all their help, time and patience during the creation of this work. I would also like to thank Katrin Tent for her helpful advise especially on the final presentation of this article. Furthermore, I am grateful for Chloé Perin's comments that made it possible to state Theorem \ref{maximaltowers} in a more general form and to simplify its proof. I would like to thank the anonymous referee for helpful comments.

\section{Bass-Serre theory and hyperbolic towers}
\label{Section Bass Serre and towers}

In this section, we collect some notions from geometric group theory needed for this article, define hyperbolic towers and give the results about them that we will use later. It follows \cite[Section 3]{LPSTowers}.
\subsection{Bass-Serre theory}
We begin with Bass-Serre theory and will only give the ideas and most important definitions. For more details, the reader is referred to \cite{Trees}.

A \emph{graph of groups} is a connected graph $\Gamma$, together with two collections of groups, $\set[G_v]_{v\in V(\Gamma)}$ (the \emph{vertex groups}) and $\set[G_e]_{e\in E(\Gamma)}$ (the \emph{edge groups}), and, for each edge $e\in E(\Gamma)$ that has endpoints $v_1$ and $v_2$, two embeddings $\alpha_e:G_e\hookrightarrow G_{v_1}$ and $\omega_e:G_e\hookrightarrow G_{v_2}$. We denote such a graph of groups by $(\mathbb{G},\Gamma)$.
To a graph of groups we can associate its \emph{fundamental group} $\pi_1(\mathbb{G},\Gamma)$. It is defined by
\begin{equation*}
\pi_1(\mathbb{G},\Gamma){\coloneqq}
\left\langle{
\begin{array}{c@{\hspace{0.5ex}:\hspace{0.5ex}}l|c@{\hspace{0.5ex}:\hspace{0.5ex}}l}
G_v &v\in V(\Gamma),& t_e^{-1}\alpha_e(g) t_e=\omega_e(g) &e\in E(\Gamma), g\in G_e,  \\
t_e&e\in E(\Gamma)\,&t_e=1 &e\in E(\Gamma_0)\\
\end{array}}\right\rangle,
\end{equation*}
where $\Gamma_0\subseteq \Gamma$ is a maximal tree in $\Gamma$. So this fundamental group consists of the elements of the vertex groups of $(\mathbb{G},\Gamma)$, together with new so-called \emph{Bass-Serre elements} $t_e$ which are introduced for each edge $e$ of $\Gamma$. The relations inside the vertex groups stay as before. Relations between elements of different vertex groups are defined by identifying images of the given embeddings up to conjugation with the corresponding $t_e$. Furthermore, whenever $e$ takes part of a fixed maximal tree $\Gamma_0$, the corresponding element $t_e$ is made trivial. The remaining non-trivial Bass-Serre elements are called \emph{Bass-Serre generators}. This means that $\pi_1(\mathbb{G},\Gamma)$ is derived from the vertex groups by a series of amalgamated products or HNN-extensions where the stable letter is the corresponding Bass-Serre generator. One can show that the isomorphism class of this fundamental group does not dependent on the choice of $\Gamma_0$. However, taking another maximal subtree changes the presentation of $\pi_1(\mathbb{G},\Gamma)$ and the choice of Bass-Serre generators.

Whenever we have a graph of groups decomposition (or \emph{splitting}) of a group $G$ (i.e. a graph of groups with fundamental group $G$), we can find a canonical action of $G$ on a simplicial tree $T$ whose quotient $G\bsl T$ is isomorphic to $\Gamma$. On the other hand, whenever $G$ acts on a simplicial tree $T$ without inversions, we get a graph of groups decomposition of $G$ with underlying graph isomorphic to $G\bsl T$. In both cases we know that vertex (respectively edge) groups of the graph of groups are conjugate to the stabilisers  of the vertices (respectively edges) of the action on $T$. In this situation, an element or a subgroup of $G$ fixing a point in $T$ is called \emph{elliptic}.

The easiest example of this is the case where $G=H\ast R$ is the free product of subgroups $H$ and $R$. In this case, the corresponding graph of groups has one edge connecting two vertices, one with vertex group $H$, the other with vertex group $R$. The edge group is trivial. If we take the same setting with a non-trivial edge group, we get an amalgamated product of $H$ and $R$.

Given an action on a tree, there several 
 different corresponding graph of groups decompositions corresponding to a choice of ``presentation''  that is defined as follows:
\begin{defi}
Let $G$ be a group acting on a tree $T$ without inversions, denote by $(\mathbb{G},\Gamma)$ the associated graph of groups and by $p$ the quotient map $p:T\to \Gamma$. A \emph{Bass-Serre presentation} for $(\mathbb{G},\Gamma)$ is a pair $(T^1,T^0)$ consisting of
\begin{itemize}
\item a subtree $T^1$ of $T$ which contains exactly one edge of $p^{-1}(e)$ for each edge $e$ of $\Gamma$;
\item a subtree $T^0$ of $T^1$  which is mapped injectively by $p$ onto a maximal sub\-tree $\Gamma_0$ of $\Gamma$.
\end{itemize}
\end{defi}

\subsection{Surface groups}
In the whole text, we assume all surfaces to be connected and compact.

It is a standard fact from the classification of surfaces that every surface $\Sigma$ is determined up to homeomorphism by its orientability, its Euler characteristic $\chi (\Sigma)$ and the number of its boundary components $b(\Sigma)$. The sphere has Euler characteristic $2$, the torus has characteristic $0$. Puncturing a surface decreases its Euler characteristic by $1$. If we decompose a surface $\Sigma$ into two surfaces $\Sigma_1$ and $\Sigma_2$, we have $\chi(\Sigma)=\chi(\Sigma_1)+\chi(\Sigma_2)$.

If $\Sigma$ is a surface with non-empty boundary, each of its boundary components has a cyclic fundamental group, which gives rise to a conjugacy class of cyclic subgroups in $\pi_1(\Sigma)$. They are called \emph{maximal boundary subgroups}. A \emph{boundary subgroup} of $\pi_1(\Sigma)$ is a non-trivial subgroup of a maximal boundary subgroup.

Let $\Sigma$ be an orientable surface with $r$ boundary components. Then $\pi_1(\Sigma)$ has a presentation of the form
\begin{gather*}
\grep{y_1,\ldots , y_{2m},s_1,\ldots ,s_r}{[y_1,y_2]\ldots [y_{2m-1},y_{2m}]=s_1\ldots s_r}
\end{gather*}
where $\chi(\Sigma)=-(2m-2+r)$ and $s_1,\ldots, s_r$ are generators of non-conjugate maximal boundary subgroups. In particular, if $\Sigma$ has non-empty boundary, we can apply a Tietze transformation by removing one of the $s_i$'s and the relation and thus get another presentation of $\pi_1(\Sigma)$ which shows that it is a free group of rank $1-\chi(\Sigma)$. This is true for non-orientable surfaces as well.

Let $\Sigma$ be a surface with non-empty boundary and $P{\coloneqq}\pi_1(\Sigma)$ its fundamental group. Let $\mathcal{C}$ be a set of 2-sided disjoint simple closed curves on $\Sigma$ that allows a collection $\set[T_c|c\in \mathcal{C}]$ of disjoint open neighbourhoods of the curves in $\mathcal{C}$ with homeomorphisms $c\times (-1,1)\to T_c$ sending $c\times\set[0]$ onto $c$. Assume in addition that no component of $\Sigma\bsl\cup\, \mathcal{C}$ has trivial fundamental group. Then we get a splitting of the group $P$ that we call the \emph{decomposition of $P$ dual to $\mathcal{C}$}. It is defined as follows: For each connected component $\Sigma_k$ of $\Sigma\bsl \bigcup_{c\in \mathcal{C}}T_c$ we get a vertex whose vertex group is $\pi_1(\Sigma_k)$. For each curve in $\mathcal{C}$ separating the components $\Sigma_k$ and $\Sigma_{k'}$, we get an edge $e_c$ with infinite cyclic edge group between the vertices corresponding to $\Sigma_k$ and $\Sigma_{k'}$ (we allow $k=k'$). Using functoriality of $\pi_1$, the inclusion maps $c\hookrightarrow \Sigma_k$ induces the embeddings of the edge groups. Such a decomposition is called the \emph{decomposition of $P$ dual to $\mathcal{C}$}. Note that here, all boundary subgroups are elliptic and edge groups are infinite cyclic. The following lemma gives a converse for this. Originally being {\cite[Theorem III.2.6.]{MSValuations}}, this version is a slight variation given in \cite{LPSTowers} as Lemma 3.2.

\begin{lem}
\label{dual decomposition}
Let $\Sigma$ be a surface with non-empty boundary and $P{\coloneqq}\pi_1(\Sigma)$ be its fundamental group. Suppose that $P$ admits a graph of groups decomposition $(\mathbb{G},\Gamma)$ in which edge groups are cyclic and boundary subgroups are elliptic. Then there exists a set $\mathcal{C}$ of disjoint simple closed curves on $\Sigma$ such that $(\mathbb{G},\Gamma)$ is the graph of groups decomposition dual to it.
\end{lem}

\subsection{Hyperbolic floors and towers}
\label{Hyperbolic towers}
\begin{defi}
\label{Definition graph of groups with surfaces}
A \emph{graph of groups with surfaces} is a graph of groups $(\mathbb{G},\Gamma)$ together with a subset $V_S$ of the vertex set $V(\Gamma)$, such that any vertex $v$ in $V_S$ satisfies:
\begin{itemize}
\item there exists a surface $\Sigma$ with non-empty boundary, such that the vertex group $G_v$ is the fundamental group $\pi_1(\Sigma)$ of $\Sigma$;
\item for each edge $e$ that has endpoint $v$, the embedding $G_e\hookrightarrow G_v$ maps $G_e$ onto a maximal boundary subgroup of $\pi_1(\Sigma)$;
\item this induces a bijection between the set of edges adjacent to $v$ and the set of conjugacy classes of maximal boundary subgroups in $\pi_1(\Sigma)$.
\end{itemize}
The vertices of $V_S$ are called \emph{surface (type) vertices} and, with a slight abuse of language, the vertex groups associated to surface type vertices are called \emph{surface (type) groups}.
The surfaces associated to the vertices of $V_S$ are called the surfaces of $(\mathbb{G},\Gamma)$.
\end{defi}

\begin{defi}
\label{Definition hyperbolic floor}
Let $(G,G',r)$ be a triple consisting of a group $G$, a subgroup $G' \leq G$ and a retraction $r$  from $G$ onto $G'$ (i.e. $r$ is a morphism $G\to G'$ which restricts to the identity on $G'$). We say that $(G,G',r)$ is a \emph{hyperbolic floor}, if there exists a graph of groups with surfaces $(\mathbb{G},\Gamma)$ with associated fundamental group $\pi_1(\mathbb{G},\Gamma)=G$ and a Bass-Serre presentation $(T^1,T^0)$ of $(\mathbb{G},\Gamma)$ such that:
 \begin{enumerate}
 \item all the surfaces of $(\mathbb{G},\Gamma)$ are either once punctured tori or have Eu\-ler cha\-rac\-te\-ris\-tic at most $-2$;
 \item $G'$ is the free product of the stabilisers of the non-surface type vertices of $T^0$;
 \item every edge of $\Gamma$ joins a surface type vertex to a non-surface type vertex;
 \item either the retraction $r$ sends surface type vertex groups of $(\mathbb{G},\Gamma)$ to non-abelian images, or the subgroup $G'$ is cyclic and there exists a retraction $r':G\ast \mathbb{Z} \to G'\ast \mathbb{Z}$ that does this.
\end{enumerate}
\end{defi}

\begin{defi}
Let $G$ be a non-cyclic group and $H\leq G$ a subgroup.
We say that $G$ is a \emph{hyperbolic tower} over $H$ (or \emph{admits a hyperbolic tower structure} over $H$), if there is a sequence of subgroups $G=G_0\geq G_1 \geq \ldots \geq G_m\geq H$ satisfying the following conditions:
\begin{itemize}
\item for each $0\leq i\leq m-1$, there exists a retraction $r_i:G_i\to G_{i+1}$ such that the triple $(G_i, G_{i+1}, r_i)$ is a hyperbolic floor and $H$ is contained in one of the non-surface type vertex groups of the corresponding graph of groups decomposition;
\item $G_m=H\ast F\ast S_1\ast\ldots\ast S_p$ where $F$ is a (possibly trivial) free group, $p\geq 0$ and each $S_i$ is the fundamental group of a closed surface of Euler characteristic at most $-2$. 
\end{itemize}
\end{defi}

\begin{figure}
\begin{center}
\includegraphics{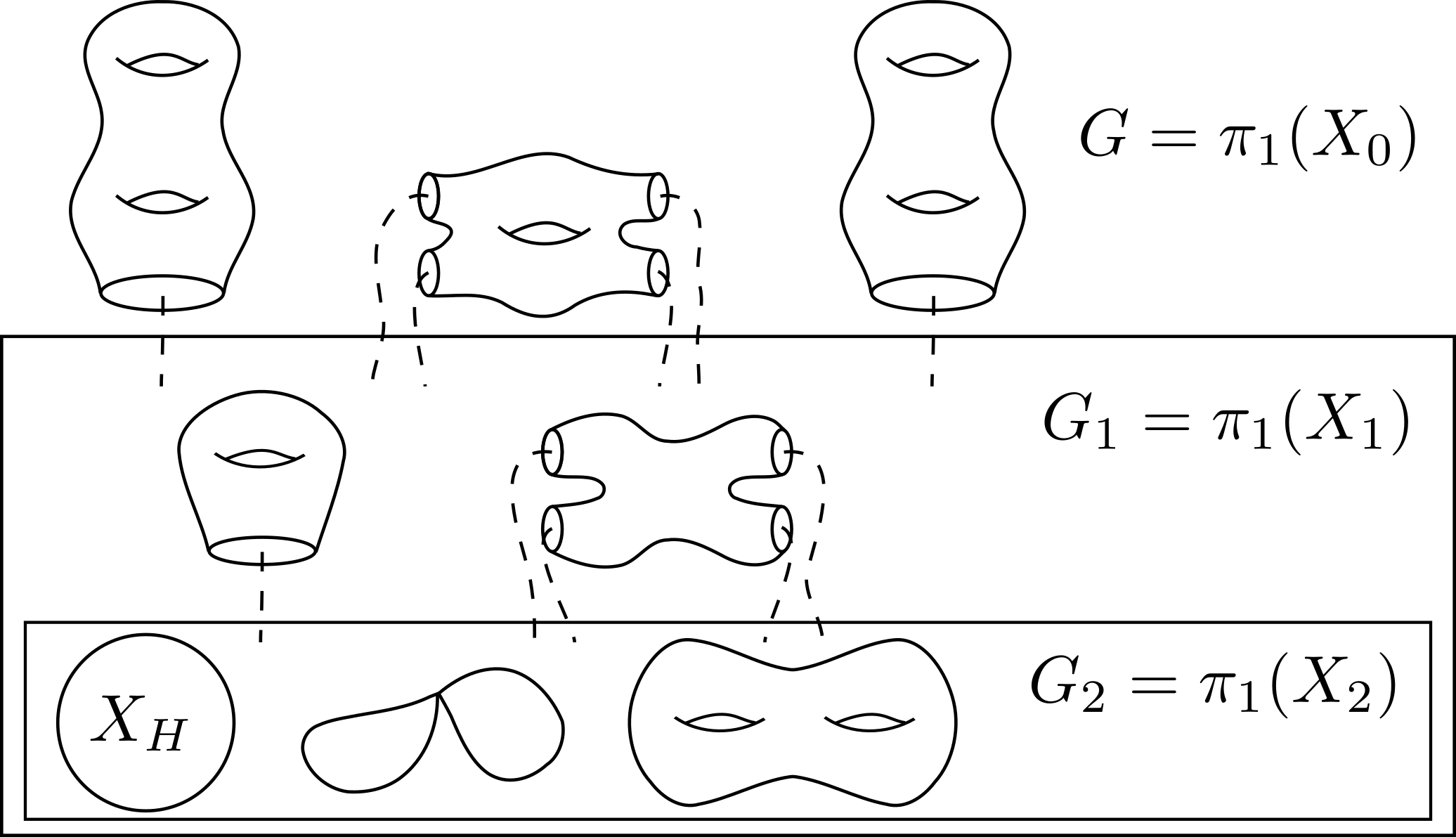}
\end{center}
\caption{A hyperbolic tower over $H$ consisting of two floors}
\label{Starting example}
\end{figure}

It is helpful to have in mind the following image of hyperbolic towers: If $G$ admits a hyperbolic tower structure over $H$, we can see $G$ as the fundamental group of a topological space $X_0$ that is derived from a space $X_H$ having fundamental group $H$ in several steps. We start with a space $X_m$ that is the disjoint union of $X_H$, a graph $X_F$ and closed surfaces $\Sigma_1,\ldots ,\Sigma_p$ of Euler characteristic at most $-2$. When $X_{i+1}$ is constructed, we get $X_{i}$ by gluing surfaces along their boundary components to $X_{i+1}$ such that there exist suitable retractions at the level of fundamental groups. 

An example of this is shown in Figure \ref{Starting example}. The nested boxes mark the sequence of subgroups of $G$. Note that although a surface represents a surface type vertex in the corresponding graph of groups, we did not mark the non-surface type vertices. The ends of the edges starting at the punctured surfaces represent the (here unspecified) points to which their boundary components are glued.

As mentioned in the introduction, hyperbolic towers gain importance for the study of $T_{fg}$ by Fact \ref{introduction trivial tower implies T_fg} which we restate here in the following way:

\begin{fact}
\label{trivial tower implies T_fg}
Let $G$ be a finitely generated group. Then $G$ is a model of $T_{fg}$ if and only if $G$ is non-abelian and admits a hyperbolic tower structure over a free subgroup.
\end{fact}
Here again, we consider the trivial group to be a free group as well. That this is equivalent to the formulation given above follows immediately from the definitions.

\section{Model-theoretic basics}
\label{Section model theory}
In this section, we will give some model theoretic basics. This will be done very briefly because although those model theoretic ideas motivate the constructions given later on, they are not needed to understand them as the definition of hyperbolic towers translates those model theoretic problems in the language of geometric group theory. For a general introduction to model theory, see for example \cite{TentZiegler}, for details about stability theory, see  \cite{PilGeomStabTheory}.

As already mentioned, we know that the common first order theory $T_{fg}$ of non-abelian free groups is stable. Stable theories enjoy a model theoretic notion of independence between elements in a given model which is called forking independence. It can be seen as a generalisation of linear independence in vector spaces and algebraic independence in algebraically closed fields, which are also basic examples for this. From now on, whenever we talk about independence, we mean forking independence. If two elements are not forking independent, we say that they ``fork'' with each other. By the results of Sela, we now can ask whether a set of elements in a free group or another model of $T_{fg}$ is independent or not.

Recall that an $m$-type over a set $A$ of a first order theory $T$ is a maximal consistent set of formulas with parameters in $A$ and at most $m$ free variables. If $G$ is a model of $T$ and $a\in G$, then the \emph{type of $a$ over $A$}, denoted by $tp(a/A)$, is the set of all formulas with parameters in $A$ satisfied by the element $a$. An important property in stable theories is the existence of \emph{generic types} over any set of parameters. A definable set $X$ of a stable group $G$ is said to be \emph{generic}, if finitely many left- (or equally right-) translates of $X$ cover $G$. A formula $\psi(x)$ is called generic, if it defines a generic set. Finally we say that a type is generic, if it contains only generic formulas. Hence we can imagine a generic type to be a type with a ``big'' set of realisations. By results of Poizat, we know that in the theory of free groups, there is a unique generic (1-)type $p_0\in S_1(T_{fg})$ over the empty set (see \cite{PilForking}), namely the type of a primitive element in a free group. 

This type is especially interesting because of the following Fact:

\begin{fact}[\cite{PilForking}]
In a finitely generated, non-abelian free group $F$, a set is a maximal independent set of realisations of $p_0$ if and only if it forms a basis of $F$.
\end{fact}

This means that at first glance, maximal independent sets of realisations of $p_0$ in non-standard models of $T_{fg}$ could be seen as analogues to bases in free groups. This is due to the fact that they look the same from the perspective of first order logic, meaning that both such sets satisfy exactly the same first order formulas and both are maximal independent with respect to forking independence. However, those sets do not necessarily generate the groups that they are taken from. Furthermore, there is no fixed size  of such a ''basis`` in a non-standard model, as we will show by proving Theorem \ref{Theorem differences in sizes}.

\section{Towers with maximal ground floors}
\label{Section maximal ground floors}
In this section, we will define maximal free ground floors and give an instruction on how to attain such floors by proving Theorem \ref{maximaltowers}. Afterwards, we will provide a model-theoretic approach to these maximal tower structures. 

\subsection{Maximal free ground floors}

\begin{defi} Let $G$ be a finitely generated model of $T_{fg}$. A subgroup $H\leq G$ is called a \emph{maximal free ground floor (in $G$)} if $H$ is free and $G$ admits a hyperbolic tower structure over $H$ but not over any other free subgroup $K\leq G$ in which $H$ is a free factor $K=H\ast H'$.
\end{defi}
Bearing in mind Fact \ref{trivial tower implies T_fg}, the fact that $G$ admits a tower structure over $H$ already implies that it is a model of $T_{fg}$. 

For the proof of Theorem \ref{maximaltowers}, we begin by collecting some lemmas about graphs of groups. The following is part of the statement of {\cite[Theorem III.2.6.]{MSValuations}} and the comments after it.

\begin{lem}
\label{Lemma subtree of action}
Let $\Sigma$ be a surface, possibly with boundary, such that $P{\coloneqq}\pi_1(\Sigma)$ acts on a tree $T$ in a way that all edge stabilisers are cyclic and boundary subgroups act elliptically. Then there is a subtree $T_0$ of $T$ that is invariant under the action of $P$ such that all edge stabilisers of the action of $P$ on $T_0$ are non-trivial, thus infinite cyclic.
\end{lem}

Using this, we deduce:
\begin{lem}
\label{indecomposabilitysurfacegroups}
Let $A_1,\ldots ,\,A_k$ be any groups and let $P\leq A_1\ast\ldots\ast A_k$ be a subgroup of their free product. Assume in addition that $P$ is the fundamental group of a surface with boundary. If every boundary subgroup of $P$ can be conjugated into some $A_i$, the group $P$ can be conjugated into one of those factors as well. If we know in addition that $P\cap A_j\not= \set[1]$, we get $P\leq A_j$.
\end{lem} 
\begin{proof}
In this situation, we know that $P$ acts on $T$, the tree associated to the free product $A_1\ast\ldots\ast A_k$, in a way that all boundary subgroups act elliptically and all edge stabilisers are trivial. So all conditions of the preceding lemma are fulfilled and we know that there is a subtree $T_0$ which is invariant under the action. If $P$ cannot be conjugated into one of the factors, this subtree cannot be trivial, so it contains at least one edge of $T$. Furthermore, the stabiliser of this edge has to be infinite cyclic which is a contradiction. Hence we know that $P\leq A_i^x$ for some $x\in A_1\ast\ldots\ast A_k$. The second part is an immediate consequence of the free product structure.
\end{proof}

\begin{lem}
\label{acylindricity}
Let $(\mathbb{G,}\Gamma)$ be a graph of groups with surfaces decomposition of a group $G$ that comes from a hyperbolic floor structure $(G, G',r)$ 
and let $T$ be the corresponding tree. Then the canonical action of $G$ on $T$ is \emph{1-acylindrical around surface type vertices}. That is, no element $g\in G\bsl \set[1]$ fixes more than one non-surface type vertex.
\end{lem}
\begin{proof}

If we have $g\in G$ that fixes two non-surface type vertices, it also fixes the shortest path between them. Since every edge of $T$ joins a surface type vertex to a non-surface type vertex, this means that $g$  fixes a segment of the tree consisting of a surface type vertex and two different edges adjacent to it. Suppose this surface type vertex is given by the coset $hP$ of the surface type vertex group $P$. With this notation, the element $g'{\coloneqq}h^{-1}g h$ fixes the vertex $(1\cdot)P$ and two different edges adjacent to it. 
Inspecting the structure of $T$, one can see that this implies that $g'\in C_e^{p_e}\cap C_{e'}^{p_{e'}}$ where $e,e'\in E(\Gamma)$ are edges in $\Gamma$ which are adjacent to the vertex corresponding to $P$.  The groups $C_e, C_{e'}\leq P$ are the maximal boundary subgroups corresponding to those edges and $p_{e},p_{e'}\in P$ are elements of the  surface group $P$. Since $g'\not=1$, this means that 
\begin{equation*}
C_e^{p_e p_{e'}^{-1}}\cap C_{e'}\not= \set[1].
\end{equation*}
Now there are two possibilities to consider:
If $e=e'$, we have $C_e^{p}\cap C_{e}\not= \set[1]$ for maximal cyclic subgroups $C_e$ and an element $p$ of $P$. As $P$ is a free group and maximal cyclic subgroups of free groups are malnormal, this implies that $C_e^{p}= C_{e}$. This contradicts the assumption that the two edges in $T$ that are fixed by $g$ are distinct.

 If on the other hand $e\not= e'$, we can conclude that the maximal boundary subgroups $C_e$ and $C_{e'}$ are conjugate in $P$. This is a contradiction, as the definition of a graph of groups with surfaces demands that different edges correspond to different conjugacy classes of maximal boundary subgroups.
\end{proof}

With this we can prove the following theorem which we want to use to construct examples of maximal free ground floors in Section \ref{Section build the towers}. We denote by $F_n$ the free group in $n$ generators.

\begin{thm}
\label{maximaltowers}
Suppose that a finitely generated group $G$ admits a hyperbolic tower structure over $H\cong F_n$ with the associated sequence of subgroups $G=G_0\geq G_1 \geq \ldots\geq G_m=H$ subject to the following conditions:
\begin{enumerate}
\item The graph of groups corresponding to the floor $(G_{i}, G_{i+1},r_i)$ consists of two vertices. One of them has vertex group $G_{i+1}$, the other one is a surface type vertex with vertex group $P_i{\coloneqq}\pi_1(\Sigma_i)$.
\item For all $i$, the surface $\Sigma_i$ is either a once punctured torus, a four times punctured sphere or a thrice punctured projective plane. 
\end{enumerate}
Then $H$ is a maximal free ground floor in $G$.
\end{thm}

\begin{proof}
Suppose that $G$ admits a second tower structure over a free subgroup $K=H\ast H'$ and take the associated sequence of subgroups to be $G=G'_0\geq G'_1 \geq \ldots\geq G'_l\geq K$.
After dividing the graph of groups decompositions of this tower structure, we may assume that for each hyperbolic floor in this structure, the associated graph of groups has only one surface type vertex. We denote by $T_j'$ the associated tree of the graph of groups decomposition corresponding to the floor $(G'_j, G'_{j+1}, r'_j)$.

First look at the top floor $(G'_0=G, G'_1, r'_{0})$. As $P_{m-1}$, the surface group that comes with the ground floor of the first tower, is a subgroup of $G$, it acts on $T_0'$. Every maximal boundary subgroup of $P_{m-1}$ can be conjugated into $H\leq G'_1$ by its corresponding Bass-Serre element and thus acts elliptically. Therefore, this action induces a splitting of $P_{m-1}$ that is by Lemma \ref{dual decomposition} dual to a set $\mathcal{C}$ of disjoint simple closed curves on $\Sigma_{m-1}$. We can assume this set of curves to be essential, i.e. no component of $\Sigma_{m-1}\bsl\cup\, \mathcal{C}$ is homeomorphic to a disc with one or no puncture.
This means that $\Sigma_{m-1}$ is decomposed into subsurfaces $(\Sigma''_k)_k$ whose fundamental groups all act elliptically on $T_0'$. 

Assume that for all $k$, the fundamental group $\pi_1(\Sigma''_k)$ stabilises a non-surface type vertex of $T_0'$. Whenever $\Sigma''_{k_1}$ and $\Sigma''_{k_2}$ are adjacent pieces of $\Sigma''$, the intersection of their fundamental groups is non-trivial. Each element contained in this intersection stabilises both the non-surface type vertex stabilised by $\pi_1(\Sigma''_{k_1})$ and the one stabilised by $\pi_1(\Sigma''_{k_2})$. Because of acylindricity (Lemma \ref{acylindricity}), these vertices have to coincide. As the surface $\Sigma''$ is connected, this shows that in fact, all the $\pi_1(\Sigma''_k)$ stabilise the same non-surface type vertex and hence, $P_{m-1}$ is conjugate to a subgroup of $G'_1$. However, at least one boundary subgroup of $P_{m-1}$ is identified with a subgroup of $H\leq G_1'$, so using acylindricity again, one has $P_{m-1}\leq G'_1$. Thus, it acts on the next floor $(G'_1, G'_2, r'_1)$. Furthermore, if  $P_{m-1}\leq G'_j$ and $t$ is a Bass-Serre generator arising in the graph of groups associated to the hyperbolic floor $(G_{m-1},G_m,r_{m-1})$, we claim that $t\in G'_j$ as well.
Indeed, if the claim is not true, there is some $k\leq j$ such that $t\in G'_{k-1}\bsl G'_k$ but $P_{m-1}\leq G'_j \leq G'_{k}$. Since the maximal boundary subgroups of $P_{m-1}$ are glued to $H$, we know that for some maximal boundary subgroup $C$ of $P_{m-1}$, we have $C^t\leq G_{m}\leq G'_{k}$. Consequently, any non-trivial element of $C$ fixes both vertices $(1\cdot)G'_k$ and $t G'_k$ in the tree corresponding to the floor $(G'_{k-1},G'_k,r'_{k-1})$. This contradicts acylindricity.

 Iterating this process we see that either $P_{m-1}\leq G'_l$, or for some $0\leq j < l$ and a subsurface $\Sigma''$ of $\Sigma_{m-1}$, the fundamental group $\pi_1(\Sigma'')$ is not included in $G'_{j+1}$
 and thus fixes a surface type vertex of $T'_j$.

Assume $P_{m-1}\leq G'_l$. As the second tower structure of $G$ is over $K$, we know that $G'_l=H\ast F\ast S_1\ast\ldots\ast S_p$ for a free group $F$ and surface groups $S_i$. All boundary subgroups of $P_{m-1}$ can be conjugated into subgroups of $G_m=H$ by their corresponding Bass-Serre element. As all those Bass-Serre elements take part of $G_l'$ as well, Lemma \ref{indecomposabilitysurfacegroups} now implies $P_{m-1}\leq H$ which is impossible.

So we know that for some $j<l$, the group $\pi_1(\Sigma'')$ fixes a surface type vertex of $T_j'$. As $\Sigma_{m-1}$ is a once punctured torus, a four times punctured sphere or a thrice punctured projective plane and the curves dividing $\Sigma_{m-1}$ are essential, we can choose $\Sigma''$ such that it contains a boundary component of $\Sigma_{m-1}$. Now using acylindricity and changing the Bass-Serre presentation of $T_j'$, we can assume that $\pi_1(\Sigma'')\leq P_j'$ where $P_j'=\pi_1(\Sigma'_{j})$ is the surface group arising in the graph of groups decomposition of the floor $(G'_j, G'_{j+1}, r'_j)$.
 
If $\pi_1(\Sigma'')$ is an infinite index subgroup of $P_j'$, we know by \cite[Lemma 3.10]{PerThesis} that $\pi_1(\Sigma'')=C_1\ast \ldots \ast C_m\ast F$ where $F$ is a (possibly trivial) free group, each $C_j$ is a boundary subgroup of $P'_j$ and any boundary element of $P'_j$ contained in $\pi_1(\Sigma'')$ can be conjugated into one of the groups $C_j$ by an element of $\pi_1(\Sigma'')$. Because the subsurface $\Sigma''\subseteq \Sigma_{m-1}$ comes from the graph of groups decomposition corresponding to the action of $P_{m-1}$ on the tree $T_j'$, we know that $\pi_1(\Sigma'')$ embeds into $P_j'$ as a surface group with boundaries. I.e. the boundary subgroups of $\pi_1(\Sigma'')$ are given by the boundary subgroups of $P_j'$ that lie in $\pi_1(\Sigma'')$.
By Lemma \ref{indecomposabilitysurfacegroups}, this implies that $\pi_1(\Sigma'')$ has to be  included completely in a boundary subgroup of $P'_j$ which is a contradiction. 

Hence, the index $n{\coloneqq}[P_j':\pi_1(\Sigma'')]$ is finite. Now, by the study of covering spaces from topology, we know that there is a covering map $p:\Sigma''\to \Sigma'_j$ of degree $n$ such that $\chi(\Sigma'')=n\cdot\chi(\Sigma'_j)$. As $\Sigma''$ is a subsurface of a once punctured torus, a four times punctured sphere or a thrice punctured projective plane, it has Euler-characteristic $-2$ or $-1$, so the index $n$ is either $1$ or $2$.

Assume that $n=2$. This can only be true if $\Sigma''$ has Euler characteristic $\chi (\Sigma'')=-2$ and $\Sigma_{j}'$ has Euler characteristic $\chi(\Sigma_{j}')=-1$. Since we assumed that the set of curves dividing $\Sigma_{m-1}$ is essential, one can deduce that in this case, $\Sigma''=\Sigma_{m-1}$ is no proper subsurface. The only surface with Euler characteristic $-1$ allowed in a hyperbolic tower structure is a once punctured torus, so we know that $\Sigma_{j}'$ has exactly one boundary component. On the other hand, it quickly follows from the definition of a covering map that 
\begin{equation*}
1=b(\Sigma_{j}')\leq b(\Sigma_{m-1})\leq n\cdot b(\Sigma_j')=2.
\end{equation*}
As we know that $b(\Sigma_{m-1})\in\set[3,4]$, this is a contradiction.

Thus we know that $n=1$, which implies that $\Sigma'_j\cong\Sigma_{m-1}$ and $P'_j=P_{m-1}$ seen as subgroups of $G$. Reordering the floors of the second tower, we may assume that $j=l-1$ which means that $G'_{l-1}$ is derived from $G'_{l}$ by gluing $\Sigma'_{j}=\Sigma_{m-1}$ to $H\leq G'_{l}$ in the same way as in the first tower.

Continuing with the action of $P_{m-2}$ on the second tower, we can apply almost the same arguments. The only thing that one needs to think about is why $P_{m-2}\leq G'_{l-1}$ is impossible. However, in the last paragraph, we assumed that $G'_{l-1}=G_{m-1}\ast F\ast S_1\ast\ldots\ast S_p$. As all boundary subgroups of $P_{m-2}$ can be conjugated into subgroups of $G_{m-1}$, we can again apply Lemma \ref{indecomposabilitysurfacegroups} to get a contradiction.

In the end of this induction process, we see that 
\begin{gather*}
G=G'_0= G_0\ast F\ast S_1\ast\ldots\ast S_p =G\ast F\ast S_1\ast\ldots\ast S_p\, ,
\end{gather*} so in particular, $F$ is trivial and we have shown the maximality.
\end{proof}

\begin{rem}
In fact, the proof shows that the second tower can be changed into the first one by permuting floors and dividing floors with several surface vertices into floors with only one surface vertex each. So up to those changes, there is only one hyperbolic tower structure of $G$ over $H$.
\end{rem}

\subsection{Model theoretic formulation}
The motivation to look at maximal free ground floors comes from the next statement: 
\begin{fact}[{\cite[Theorem 7.1]{RizosHohogeneity}}]
\label{independent realizations because of tower}
Let $G$ be a non-abelian finitely generated group. Then $k$ elements $u_1,\ldots ,u_k$ of $G$ form an independent set of realisations of $p_0$ if and only if $H{\coloneqq}\langle u_1,\ldots ,u_k \rangle\leq G$ is free of rank $k$ and $G$ admits a hyperbolic tower structure over $H$.
\end{fact}

This immediately implies the following corollary:

\begin{cor}
\label{Corollary maximal ground floors and p_0}
A subgroup $H\cong F_n$ of a finitely generated group $G$ is a maximal free ground floor in $G$ if and only if each basis of $H$ is a maximal independent set of realisations of $p_0$.
\end{cor}

So the generators of maximal free ground floors are exactly the analogues to bases mentioned at the end of Section \ref{Section model theory}. However, in contrast to free groups, not all such ``bases'' of a fixed non-standard model of $T_{fg}$ have the same cardinality. 
The fact that the ratios between the basis lengths of two such subgroups can even get arbitrarily large is what we will prove in the next section.

\section{Constructing such towers}
\label{Section build the towers}
In this section, we give examples of models of $T_{fg}$ that each contain  maximal free ground floors of different basis lengths. 

The ideas of this are taken from \cite[Proposition 5.1]{LPSTowers} which now can be seen as the special case $n=1$ of Theorem \ref{alternative proof for infinite weight}.
 
\subsection{A special case with pictures}
\label{2s5s}
To start with and in order to explain the idea, we will construct a group that contains one maximal free ground floor of basis length $2$ and one of basis length $5$. Doing this, we will emphasise the geometric motivation and give a more technical proof in the general case afterwards.

At first, we look at two hyperbolic floors that we will use during the construction. Let $H$ be any non-abelian group, $\Sigma$ a four times punctured sphere and $\Sigma'$ a once punctured torus. We describe the hyperbolic floors by their decompositions as graphs of groups with surfaces. In both cases, we have one non-surface type vertex with vertex group $H$ and one surface type vertex with vertex group $\pi_1(\Sigma)$ (respectively $\pi_1(\Sigma')$). As required by the definitions, the edge groups of these graphs of groups with surfaces are identified with maximal boundary subgroups.

\subsubsection{Gluing $\Sigma$ to $H$}
\label{gluing 3 holes}

We know that there is a presentation
\begin{equation*}
\pi_1(\Sigma)=\grep{s_1,s_2,s_3,s_4}{s_1s_2s_3s_4=1} ,
\end{equation*}
where the $s_i$'s are generators of non-conjugate maximal boundary subgroups of $\pi_1(\Sigma)$. As there are four conjugacy classes of maximal boundary subgroups, the  two vertices of the graph of groups we describe are connected by four edges. Thus, we get three Bass-Serre generators $t_1, t_2, t_3$. The embeddings of the edge groups into $H$ are given by identifying  
\begin{equation*}
t_1^{-1}s_1t_1=w_1,\, t_2^{-1}s_2t_2=w_1^{-1},\, t_3^{-1}s_3t_3=w_2,\, s_4=w_2^{-1}
\end{equation*} for any two non-commuting elements $w_1,w_2\in H$. The result is the group
\begin{gather*}
G{\coloneqq}\grep{H, t_1,t_2,t_3}{t_1 w_1 t_1^{-1}t_2 w_1^{-1} t_2^{-1}=[w_2,t_3]}. 
\end{gather*}
If one looks at the retraction given by
\begin{eqnarray}
\nonumber
r:G & \to & H  \\
t_1,t_2,t_3   & \mapsto & 1
\nonumber\, ,
\end{eqnarray}
$\pi_1(\Sigma)$ is sent to $\langle w_1, w_2\rangle\leq H$. Thus, the tuple $(G,H,r)$ is a hyperbolic floor.

\subsubsection{Gluing $\Sigma'$ to $H$}
\label{gluing one hole}
Writing 
\begin{equation*}
\pi_1(\Sigma')=\langle y_1, y_2, s |  [y_1,y_2]=s\rangle ,
\end{equation*}
the element $s$ is a generator of a maximal boundary subgroup. Identifying $s$ with the commutator $[w_1,w_2]$ for any non-commuting elements $w_1,w_2\in H$, we get the group
\begin{equation*}
G'{\coloneqq}\grep{ H, y_1, y_2}{ [y_1,y_2]=[w_1,w_2]}. 
\end{equation*}
By adding the retraction
\begin{eqnarray}
\nonumber
r':G' & \to & H  \\
\nonumber
y_1   & \mapsto & w_1\\
\nonumber
y_2   & \mapsto & w_2,
\end{eqnarray}
we get a hyperbolic floor $(G',H,r')$. 
\newline

We now use these two kinds of floors to construct a group with different maximal tower structures.

\begin{thm}
\label{Theorem 2 and 5}
The group 
\begin{gather*}
G{\coloneqq}
\left\langle{
\begin{array}{c|l@{\hspace{0.5ex}=\hspace{0.5ex}}l}
a_1,a_2,t_1,t_2,t_3,  & t_1 a_1 t_1^{-1}t_2 a_1^{-1} t_2^{-1}&[a_2,t_3], \\
t_4,t_5,t_6,  & t_4 a_2 t_4^{-1}t_5 a_2^{-1} t_5^{-1}&[t_1^{-1},t_6], \\
t_7,t_8,t_9 & t_7 t_1^{-1} t_7^{-1}t_8 t_1 t_8^{-1}&[t_4^{-1},t_9] \\
\end{array}} \right\rangle
\end{gather*}
is a model of $T_{fg}$ and contains maximal free ground floors of basis lengths $2$  and $5$.
\end{thm}
\begin{proof}
By Fact \ref{trivial tower implies T_fg}, the fact that $G$ contains some maximal free ground floor
already implies that $G$ has the same theory as a free group, so it suffices to describe such tower structures of $G$.

\begin{figure}
\begin{center}
\includegraphics{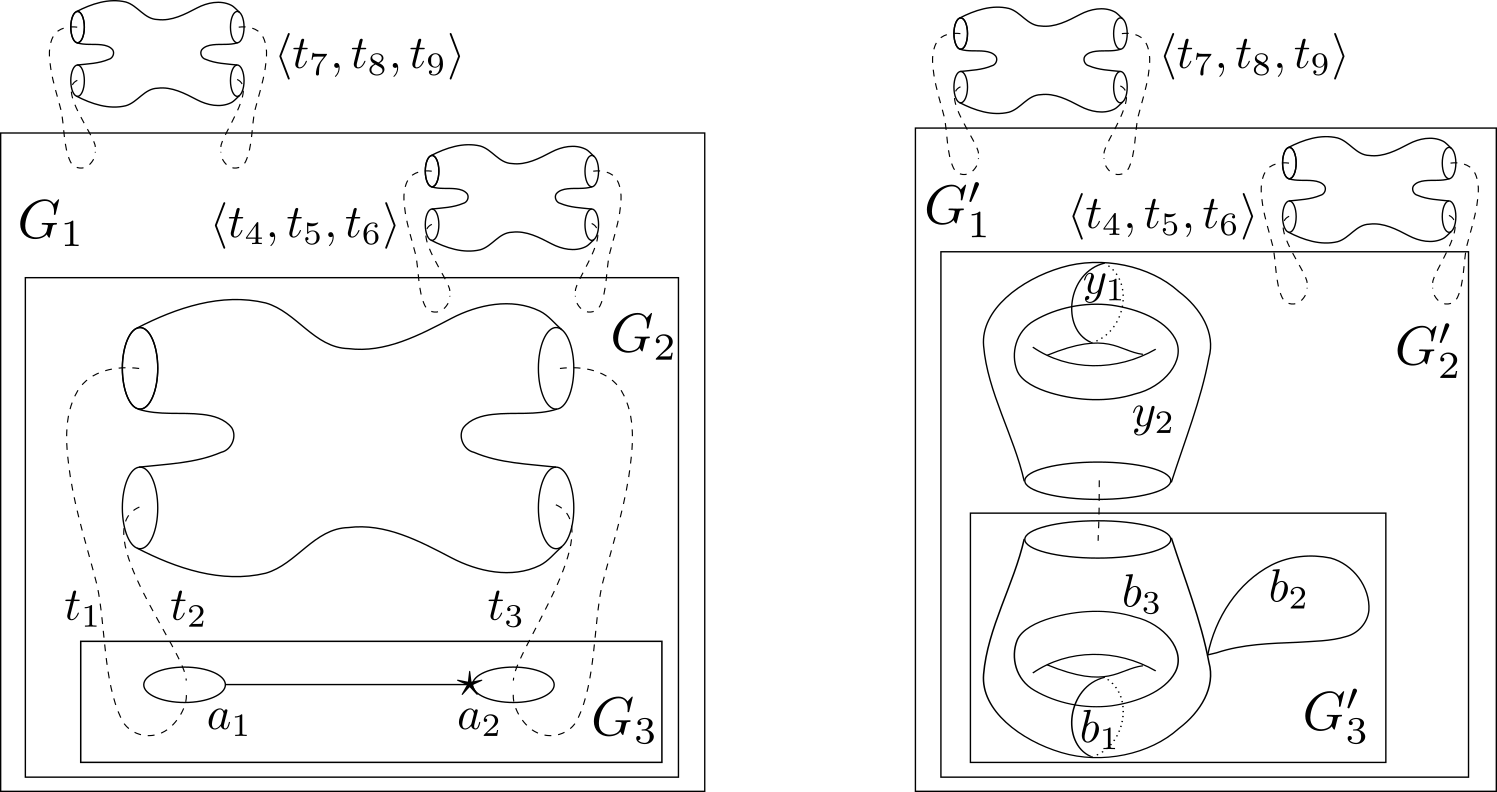}
\end{center}
\caption{The tower structures of $G$ and $G'$; on the left the basepoint $\star$ matching the isomorphism $f:G\to G'$ is marked}
\label{Firststep in example}
\end{figure}

We begin by observing that $G$ has a hyperbolic tower structure over $\langle a_1, a_2\rangle\cong F_2$ consisting of three floors of the form $G=G_0\geq G_1\geq G_2\geq G_3=\langle a_1, a_2\rangle$. It is illustrated on the left of Figure \ref{Firststep in example}.
  In all of the three floors, the corresponding graph of groups decomposition of $G_i$ consists of two vertices: one vertex with vertex group $G_{i+1}$ and one surface type vertex where the surface $\Sigma_i$ that is added is a four times punctured sphere which is glued along its boundary to $G_{i+1}$ as described in \ref{gluing 3 holes}. Firstly, $G_2=\langle a_1, a_2, t_1, t_2, t_3\rangle$ is derived from $G_3$ by gluing the boundary components of $\Sigma_2$ to $a_1, a_1^{-1}, a_2$ and $a_2^{-1}$ (i.e. choosing $w_1=a_1$ and $w_2=a_2$) and adding Bass-Serre generators $t_1, t_2$ and $t_3$ for the first three gluings. From this, $G_1= \ll a_1, a_2, t_1, \ldots ,t_6 \rr$ is derived by gluing the boundary components of $\Sigma_1$ to $a_2, a_2^{-1}, t_1^{-1}$ and $t_1$. Here, the Bass-Serre generators $t_4,t_5$ and $t_6$ are added. Lastly, we obtain $G=G_0$ from $G_1$ by gluing $\Sigma_0$ to $G_1$, identifying the sphere's maximal boundary subgroups with the groups generated by $t_1^{-1},t_1, t_4^{-1}$ and $t_4$  and adding Bass-Serre generators $t_7,t_8$ and $t_9$. So the sequence of subgroups is given by the different lines in the presentation above. 
   Although here, it is quite easy to believe that in all floors, the ``gluing points'' do not commute, this will be less obvious in the general case, so we check it now to explain how one can verify this. For the first floor $G_2\geq G_3$, it is clear that $a_1$ and $a_2$ do not commute as they form a basis of $G_3$. We will show later that the other floors fulfil this condition as well.

Clearly, the conditions of Theorem \ref{maximaltowers} are satisfied. Thus, we know that $\set[a_1,a_2]$ is a maximal independent set of realisations of $p_0$.
To get such sets of other sizes, we will step by step change the geometric interpretation of the hyperbolic floors.

At first, we observe that the first floor $G_2\geq G_3$ can be interpreted as a decomposition of a double torus with one arc. We imagine the handles of this double torus to be cut such that $G_3$ can be seen as the fundamental group of two loops connected by the arc and $G_2$ is derived from this by gluing the rest of the double torus (a four times punctured sphere) to it  (see Figure \ref{Firststep in example} on the left). On the other hand, there is another decomposition of this object that can be interpreted as a hyperbolic floor $G_2'\geq G_3'$: Here we cut the double torus between the two handles such that we gain a once punctured torus with an arc whose fundamental group is $G_3'$. What remains is another once punctured torus whose fundamental group is the surface type vertex group in this hyperbolic floor (as shown on the right of Figure \ref{Firststep in example}). This implies that $G$ admits as well a presentation $G'$ of the following form:
\begin{equation*}
G'{\coloneqq}
\left\langle{
\begin{array}{c|l}
b_1,b_2,b_3,y_1,y_2, & [y_1,y_2]=[b_1,b_3], \\
t_4,t_5,t_6, & t_4 b_1 t_4^{-1}t_5 b_1^{-1} t_5^{-1}=[b_2,t_6], \\
t_7,t_8,t_9 & t_7 b_2 t_7^{-1}t_8 b_2^{-1} t_8^{-1}=[t_4^{-1},t_9] \\
\end{array}}\right\rangle.
\end{equation*}
The isomorphism $f:G\to G'$ is given by
\begin{eqnarray}
\nonumber
f:G & \to & G'  \\
\nonumber
a_1   & \mapsto & b_2y_1b_2^{-1}\\
\nonumber
a_2   & \mapsto & b_1\\
\nonumber
t_1   & \mapsto & b_2^{-1}\\
\nonumber
t_2   & \mapsto & y_2b_2^{-1}\\
\nonumber
t_3   & \mapsto & b_3
\end{eqnarray}
and the identity on the other generators. This isomorphism $f$ sends the images of our gluing points for $\Sigma_1$ in the original tower structure of $G$ to $f(a_2)=b_1$ and $f(t_1^{-1})=b_2$. As $b_1$ and $b_2$ take part of a basis of $G'_3$, we see that $a_2$ and $t_1^{-1}$  do  not commute in $G$.

\begin{figure}
\begin{center}
\includegraphics{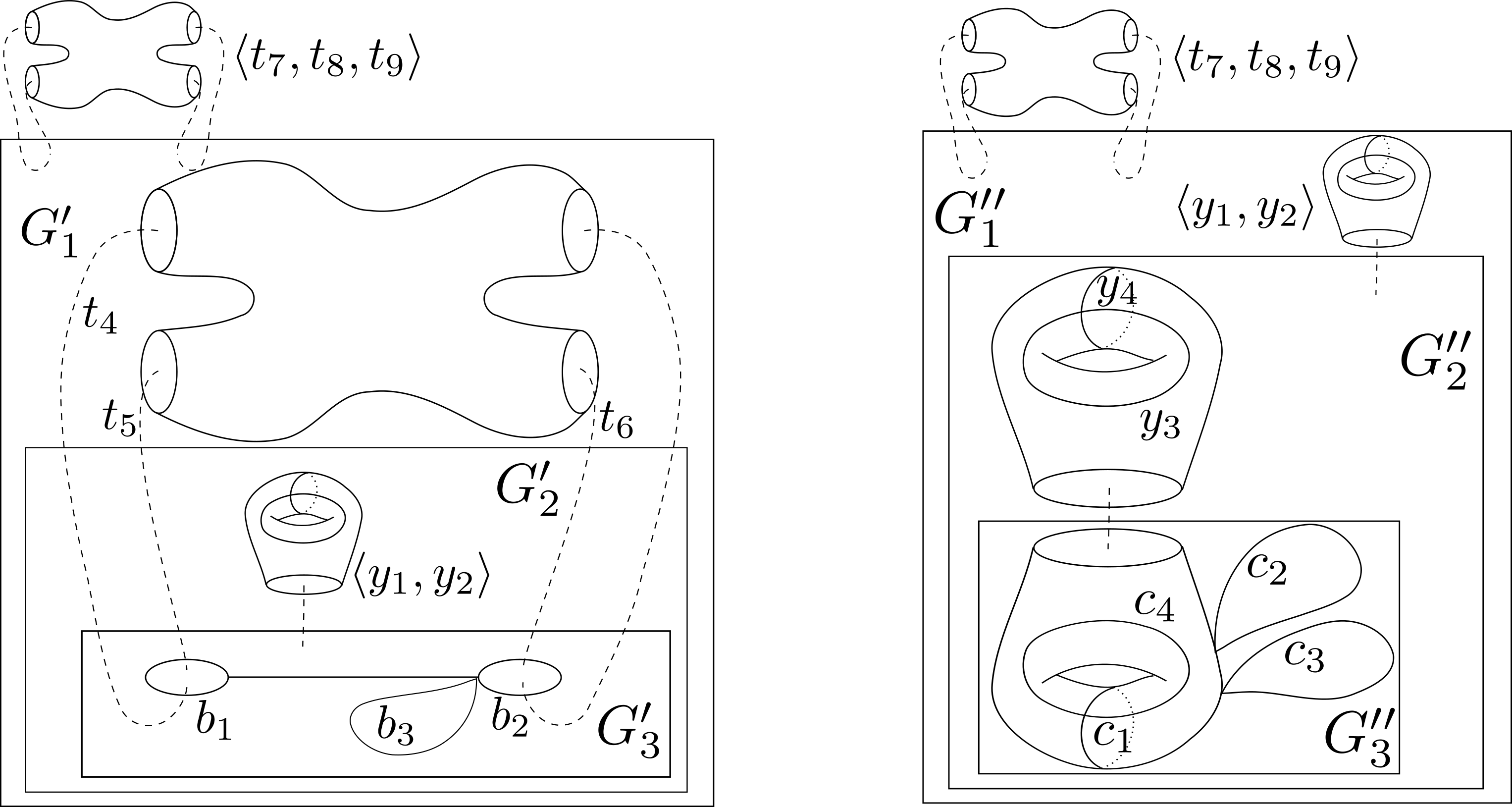}
\end{center}
\caption{The reinterpreted tower structure of $G'$ and the one of $G''$}
\label{Second step in example}
\end{figure}

Now we change the geometric interpretation of $G_3'=\ll b_1, b_2, b_3\rr$ and see it as the fundamental group of two loops which are connected by an arc and have corresponding generators $b_1$ and $b_2$ together with a third loop represented by $b_3$ (see Figure \ref{Second step in example} on the left).
Since the four times punctured sphere of the floor $G'_1\geq G_2'$ is now, similar to the case above, glued to the first two loops, we can apply the same procedure to get an isomorphism $f':G'\to G''$ onto the group
\begin{gather*}
G''{\coloneqq}
\left\langle{
\begin{array}{c|l}
c_1,c_2,c_3,c_4, y_1,y_2, & [y_1,y_2]=[c_2y_3c_2^{-1},c_3], \\
y_3,y_4, & [y_3,y_4]=[c_1,c_4], \\
t_7,t_8,t_9 & t_7 c_1 t_7^{-1}t_8 c_1^{-1} t_8^{-1}=[c_2,t_9] \\
\end{array}}\right\rangle
\end{gather*}
(see Figure \ref{Second step in example}). The images of the gluing points for $\Sigma_0$ are $f'(f(t_1^{-1}))=c_1$ and $f'(f(t_4^{-1}))=c_2$. So we see that $t_1^{-1}$ and $t_4^{-1}$ do not commute in $G$.

Doing the same reinterpretation process a third time,  we see that $G$ is isomorphic to
\begin{gather*}
G'''{\coloneqq}
\left\langle{
\begin{array}{c|l}
d_1,d_2,d_3,d_4,d_5, y_1,y_2,  & [y_1,y_2]=[d_1 y_3 d_1^{-1},d_3], \\
y_3,y_4, & [y_3,y_4]=[d_2y_5d_2^{-1},d_4], \\
y_5,y_6 & [y_5,y_6]=[d_1,d_5] \\
\end{array}}\right\rangle
.
\end{gather*}
$G'''$ now has a hyperbolic tower structure $G'''=G'''_0\geq G'''_1\geq G'''_2\geq G'''_3=\langle d_1,\ldots d_5 ,\rangle$ with
\begin{align*}
G'''_1=\langle d_1, \ldots ,d_5, y_3,y_4,y_5,y_6\rangle \leq G''',& & G'''_2=\langle d_1, \ldots ,d_5, y_5, y_6 \rangle\leq G'''.
\end{align*}
Here, for all $i$, the corresponding graph of groups decomposition of $G'''_i$ consists of two vertices, one with vertex group $G'''_{i+1}$ and one surface type vertex where the surface is a once punctured torus. All those tori are glued to the floor below as described in \ref{gluing one hole}, their gluing points are
\begin{align*}
[d_1 y_3 d_1^{-1},d_3]&\in G'''_1,&[d_2y_5d_2^{-1},d_4]&\in G'''_2&\text{and}&&[d_1,d_5]&\in G'''_3.
\end{align*}
The only thing that one has to check is whether all those commutators are non-trivial. But this can be shown in the same way as it was done for the gluing points of the four times punctured spheres. 

Hence, Theorem \ref{maximaltowers} tells us that $\langle d_1,\ldots ,d_5\rangle$ is a maximal free ground floor in $G'''$ and taking its preimage, we find such a subgroup in $G$, too.
\end{proof}

\subsection{The general case}
Now, we generalise the result of the last subsection to arbitrarily large ratios between the basis lengths of the ground floors. We start with the following technical proposition:

\begin{prop}
The group
\label{presentations}
\[
\sbox0{\ensuremath{
\begin{array}{c|c}
a_1,a_2,t_1,t_2,t_3, & t_1 w_1 t_1^{-1}t_2 w_1^{-1} t_2^{-1}=[w_2,t_3], \\
t_4,t_5,t_6, & t_4 w_3 t_4^{-1}t_5 w_3^{-1} t_5^{-1}=[w_4,t_6], \\
\vdots& \vdots \\
t_{3n-2},t_{3n-1},t_{3n} & t_{3n-2} w_{2n-1} t_{3n-2}^{-1}t_{3n-1} w_{2n-1}^{-1} t_{3n-1}^{-1}=[w_{2n},t_{3n}] \\
\end{array}}}
\mathopen{G^n{\coloneqq} \resizebox{1.2\width}{1.1\ht0}{$\Bigg\langle$}}
\raisebox{2pt}{\usebox{0}}
\mathclose{\resizebox{1.2\width}{1.1\ht0}{$\Bigg\rangle$}}
\]
with 
\begin{gather}
w_1{\coloneqq}a_1,\, w_2=w_3{\coloneqq}a_2, \nonumber \\
w_{2i+2}=w_{2i+3}{\coloneqq}t^{-1}_{3i-2} \text{ for $i\geq 1$,}
 \nonumber
\end{gather}
admits a presentation of the form
\[
\sbox0{\ensuremath{
\begin{array}{c|c}
e_1,\ldots , e_{n+2}, y_1,y_2, & [y_1,y_2]=[w'_1,w'_2], \\
y_3,y_4, & [y_3,y_4]=[w'_3,w'_4], \\
\vdots  & \vdots \\
y_{2n-1},y_{2n} & [y_{2n-1},y_{2n}]=[w'_{2n-1},w'_{2n}] \\
\end{array}}}
\mathopen{{\tilde{G}^n}= \resizebox{1.2\width}{1.1\ht0}{$\Bigg\langle$}}
\raisebox{2pt}{\usebox{0}}
\mathclose{\resizebox{1.2\width}{1.1\ht0}{$\Bigg\rangle$}}
\]
where for all $1\leq j\leq n$, the words $w'_{2j-1}$ and $w'_{2j}$ are 
elements in the subgroup generated by $e_1,\ldots ,e_{n+2},y_{2j+1},\ldots , y_{2n} $.
\end{prop}
\begin{proof}
We give a sequence of isomorphisms $(f^{(i)}:G^{(i)}\to G^{(i+1)})_{0\leq i\leq n-1}$ where $G^{(0)}=G^n$, $G^{(n)}=\tilde{G}^n$ and, for $0< i < n$, the group $G^{(i)}$ is defined by the following presentation:
\[
\sbox0{${
\begin{array}{c|c}
a^{(i)}_1,\ldots , a^{(i)}_{i+2}, y_1,y_2,  & [y_1,y_2]=[f_{i-1}(w_2),a^{(i)}_3], \\
y_3,y_4, & [y_3,y_4]=[f_{i-1}(w_4),a^{(i)}_4], \\
\vdots & \vdots \\
y_{2i-1},y_{2i},& [y_{2i-1},y_{2i}]=[f_{i-1}(w_{2i}),a^{(i)}_{i+2}], \\
t_{3i+1},t_{3i+2},t_{3i+3},& t_{3i+1} \mathbf{a^{(i)}_1} t_{3i+1}^{-1}t_{3i+2} \mathbf{(a^{(i)}_1)^{-1}} t_{3i+2}^{-1}=[\mathbf{a^{(i)}_2},t_{3i+3}],\\
t_{3i+4},t_{3i+5},t_{3i+6}, & t_{3i+4} \mathbf{a^{(i)}_2} t_{3i+4}^{-1} t_{3i+5} \mathbf{(a^{(i)}_2)^{-1}} t_{3i+5}^{-1}=[\mathbf{t^{-1}_{3i+1}},t_{3i+6}],\\
t_{3i+7},t_{3i+8},t_{3i+9}, &
t_{3i+7} \mathbf{t^{-1}_{3i+1}} t_{3i+7} t_{3i+8} \mathbf{t_{3i+1}} t_{3i+8}^{-1}=[\mathbf{t_{3i+4}^{-1}},t_{3i+9}],\\
\vdots & \vdots \\
t_{3n-2},t_{3n-1},t_{3n} & t_{3n-2} w_{2n-1} t_{3n-2}^{-1}t_{3n-1} w_{2n-1}^{-1} t_{3n-1}^{-1}=[w_{2n},t_{3n}]
\end{array}}$
}
\mathopen{\resizebox{1.2\width}{1.13\ht0}{$\Bigg\langle$}}
\raisebox{9pt}{\usebox{0}}
\mathclose{\resizebox{1.2\width}{1.13\ht0}{$\Bigg\rangle$}}
\]
where $f_i{\coloneqq}f^{(i)}\circ f^{(i-1)}\circ\ldots \circ f^{(0)}$ and bold letters mark some images of the $w_j$'s that are important to understand this step.
The isomorphisms are defined by
\begin{eqnarray}
\nonumber
f^{(i)}:G^{(i)}&\to &G^{(i+1)}  \\
\nonumber
a^{(i)}_1   & \mapsto & a^{(i+1)}_2 y_{2i+1} (a^{(i+1)}_2)^{-1}\\
\nonumber
a^{(i)}_2   & \mapsto & a^{(i+1)}_1\\
\nonumber
a^{(i)}_3   & \mapsto & a^{(i+1)}_3\\
\nonumber
&
\vdots&\\
\nonumber
a^{(i)}_{i+2}   & \mapsto & a^{(i+1)}_{i+2}\\
\nonumber
t_{3i+1}   & \mapsto & (a^{(i+1)}_2)^{-1}\\
\nonumber
t_{3i+2}   & \mapsto & y_{2i+2} (a^{(i+1)}_2)^{-1}\\
\nonumber
t_{3i+3}   & \mapsto & a^{(i+1)}_{i+3}
\end{eqnarray}
and the identity on the remaining generators. (In the cases $i=0$ respectively $i=n-1$, we take $a^{(0)}_j{\coloneqq}a_j$ and $a^{(n)}_j{\coloneqq}e_j$.)
Since we find a preimage for every ge\-ne\-ra\-tor of $G^{(i+1)}$, the map $f^{(i)}$ is surjective. We have
\begin{gather}
f^{(i)}(t_{3i+1} a^{(i)}_1 t_{3i+1}^{-1}t_{3i+2} (a^{(i)}_1)^{-1} t_{3i+2}^{-1})=[y_{2i+1},y_{2i+2}], \nonumber\\
f^{(i)}(t_{3i+1}^{-1})=(a^{(i+1)}_2),\nonumber
\end{gather}
and $f^{(i)}$ fixes all $t_j$ with $j>3i+3$. This shows that each relation in $G^{(i+1)}$ corresponds to exactly one relation in $G^{(i)}$, which one can use to show that $f^{(i)}$ is a well-defined homomorphism and injective. 

Defining 
\begin{align*}
w'_{2j}&{\coloneqq}f_{n-1}(t_{3j})=e_{j+2}  ,\\
w'_{2j-1}&{\coloneqq}f_{n-1}(w_{2j}) ,
\end{align*}
it follows that $f_{n-1}=f^{(n-1)}\circ f^{(n-2)}\circ\ldots \circ f^{(0)}$ is an isomorphism between $G^n$ and $\tilde{G}^n$.

It remains to show that $w'_{2j-1}$ and $w'_{2j}$ lie in the subgroup generated by $e_1, \ldots , e_{n+2}$, $y_{2j+1}, \ldots , y_{2n}$. A short computation shows that $f_{j-1}(w_{2j})
=a_1^{(j)}$, so the smallest index of any instance of $y_k$ appearing in $f_{n-1}(w_{2j})$ is $k=2j+1$. As we already know that $w'_{2j}=e_{j+2}$, this finishes the proof.
\end{proof}

Using this proposition, we can finally show the following which proves Theorem \ref{Theorem differences in sizes}:
\begin{thm}
\label{alternative proof for infinite weight}
The group $G^n$ as defined in Proposition \ref{presentations} is a model of $T_{fg}$ that contains maximal free ground floors of basis lengths $2$ and $n+2$.
\end{thm}
\begin{proof}
We will describe two hyperbolic tower structures of $G^n$ over free subgroups. As in the special case of Theorem \ref{Theorem 2 and 5}, the existence of such structures immediately implies  that $G^n$ is a model of $T_{fg}$.

The first structure is over $\langle a_1, a_2\rangle\cong F_2$ and consists of $n$ floors. The associated sequence of subgroups is given by $G^n=G_0\geq G_1\geq \ldots\geq G_n=\langle a_1, a_2\rangle$ where $G_j$ is generated by $a_1,a_2,t_1,\ldots , t_{3(n-j)}$. For all floors, the corresponding graph of groups decomposition of $G_j$ consists of two vertices: one vertex with vertex group $G_{j+1}$ and one surface vertex where the surface $\Sigma_{j}$ added is a four times punctured sphere that is glued to $G_{j+1}$ as in \ref{gluing 3 holes}. That is, the maximal boundary subgroups of $\pi_1(\Sigma_j)$ are identified with $w_{2(n-j)-1},\,w_{2(n-j)-1}^{-1},\, w_{2(n-j)}$ and $w_{2(n-j)}^{-1}$, which are all elements of $G_{j+1}$. Doing so, we have to add the Bass-Serre generators $t_{3(n-j)-2},t_{3(n-j)-1}$ and $t_{3(n-j)}$. 
As in the proof of Theorem \ref{Theorem 2 and 5}, we can deduce from the proof of Proposition \ref{presentations} that $w_{2(n-j)-1}$ and $w_{2(n-j)}$ do not commute. This shows that these decompositions describe hyperbolic floors that satisfy all the conditions of Theorem \ref{maximaltowers}. Consequently, we know that $\langle a_1,a_2\rangle$ is a maximal free ground floor.

On the other hand, Proposition \ref{presentations} tells us that $G^n\cong\tilde{G}^n$ and $\tilde{G}^n$ admits a hyperbolic tower structure over $\langle e_1,\ldots,e_{n+2}\rangle\cong F_{n+2}$. Just like the first decomposition, it consists of $n$ floors where each associated graph of groups has one non-surface type vertex and one surface type vertex. Here, all the surfaces are once punctured tori denoted by $\tilde{\Sigma}_j$ and they are glued to the floors below as described in \ref{gluing one hole}. The corresponding sequence of subgroups is $\tilde{G}^n=\tilde{G}_0\geq \tilde{G}_1\geq \ldots\geq \tilde{G}_n=\langle e_1,\ldots,e_{n+2}\rangle$ where $\tilde{G}_j$ is the subgroup generated by $e_1,\ldots ,\,e_{n+2},\,y_{2j+1},\ldots ,\, y_{2n} $. In the floor $\tilde{G}_j\geq\tilde{G}_{j+1}$, a maximal boundary subgroup of $\pi_1(\tilde{\Sigma}_j)$ is identified with the commutator $[w'_{2j+1},w'_{2j+2}]$ that takes by Proposition \ref{presentations} part of  $\tilde{G}_{j+1}$. This tower satisfies all conditions of Theorem \ref{maximaltowers}, so we know that $\langle e_1,\ldots,e_{n+2}\rangle$ is a maximal free ground floor as well.
\end{proof}
\begin{rem}
In fact the proof of Proposition \ref{presentations} shows that $G^n$ even contains maximal free ground floors of all basis lengths between $2$ and $n+2$ because for each $0\leq i\leq n$, the group $G^{(i)}$ admits a hyperbolic tower over $F_{i+2}$ that fulfils the conditions of Theorem \ref{maximaltowers}.
\end{rem}

\section{Weight and Whitehead graphs}
\label{Section infinite weight}

In this last section, we want to take a closer look at the model theoretic meaning of the results presented so far and prove Theorem \ref{Theoreminfiniteweight}. These model theoretic questions were the point of departure for this article.

The analogies of forking independence to classical independence notions as linear independence or algebraic independence lead to the idea of comparing the sizes of maximal independent sets of realisations of a fixed type. That is the motivation for introducing the so-called weight of a type which bounds the ratio of the sizes of different such sets.

\begin{defi}
The \emph{preweight} of a type $p(\bar{x}){\coloneqq}tp(\bar{a}/A)$ is the supremum of the set of cardinals $\kappa$ for which there exists a set $\set[\bar{b}_i|i<\kappa]$ independent over $A$, such that $\bar{a}$ forks with $\bar{b}_i$ over $A$ for all $i$. It is denoted by $\text{prwt}(q)$. The \emph{weight} $wt(p)$ of a type $p$ is the supremum of 
\begin{equation*}
\set[\text{prwt}(q)|q \text{ a non-forking extension of }p].
\end{equation*}
\end{defi}

In fact, for every type $p$ in a theory $T$, the weight $wt(p)$ is smaller or equal to the cardinality of $T$. In our case, where we consider the countable theory $T_{fg}$ of free groups, the weight of all types is bounded by $\omega$.

The mentioned bound to the ratio of maximal independent sets is given by the following:

\begin{fact}[see \cite{MakkaiStability} and {\cite[Conclusion V.3.13]{ShelahClass}}]
\label{fact about weight}
Let $T$ be a complete theory. Suppose $p$ is a type in $T$ such that $wt(p)\leq n$ for a natural number $n\in\mathbb{N}$. Then we can find no model of $T$ in which there exist two maximal independent sets of realisations of $p$, such that one has size $k$ while the other one has size greater than $k\cdot n$. 
\end{fact}
In particular, if $wt(p)=1$, we know that all such sets have the same size.

Bearing in mind that each basis of a maximal free ground floor forms a maximal independent set of realisations of the generic type $p_0$ (see Corollary \ref{Corollary maximal ground floors and p_0}), one can also see Theorem \ref{alternative proof for infinite weight} as a proof that $p_0$ has infinite weight. This is a fact that was already shown by Pillay (\cite{PilGenweight}) and Sklinos (\cite{Rizinfweight}). 

Extending the methods of Sklinos' proof, we now want to generalise this result to arbitrary types realised in free groups. More precisely, we show that any non-algebraic (1-)type over the empty set which is realised in a free group has infinite weight. The condition on the types to be non-algebraic is no great restriction as in $T_{fg}$, all (1-)types over the empty set but the one of the neutral element are non-algebraic.

For this, we will use the following strong answer to the Tarski problem given by Sela:
\begin{fact}[{\cite{SelaDiophantine}}]
\label{Answer Tarski problem}
For any $2\leq m\leq n$, the natural embedding of $F_m$ in $F_n$ is elementary.
\end{fact}

From now on, we will denote by $F_n$ the free group generated by the set $X{\coloneqq}\set[e_1, \ldots , e_n]$. We call a word $w=u_1 u_2 \ldots u_k$ with $u_i\in X\cup X^{-1}$ \emph{reduced}, if it contains no subword of the form $u u^{-1}$. We say that $w$ is \emph{cyclically reduced}, if it cyclically contains no such subword, that is neither $w$ nor any cyclical permutation of its letters contain a subword $u u^{-1}$. With this notation, $F_n$ can be seen as the set of reduced words over $X$ with multiplication given by concatenation of words followed by reductions.

\begin{defi}
Let $A\subseteq F_n$ be a set of elements in $F_n$. Then $A$ is called \emph{separable}, if there exists a non-trivial free decomposition $F_n=G\ast H$, such that each element of $A$ can be conjugated either into $G$ or into $H$. This means that for each $a\in A$, there exists $x\in F_n$ such that $xax^{-1}\in G\cup H$. 
\end{defi}

The connection between separability and independence in free groups is established by the following result from \cite{ForkingandJSJ} that characterises independence in free groups by the possibility to find proper free decompositions.

\begin{fact}[{\cite[Theorem 1]{ForkingandJSJ}}]
\label{separability and independence}
Let $\bar{u}, \bar{v}$ be tuples of elements in  the free group with $n$ generators and let $S$ be a free factor of $F_n$. Then $\bar{u}$ and $\bar{v}$ are independent over $S$ if and only if $F_n$ admits a free decomposition $F_n=G\ast S\ast H$  with $\bar{u}\in G\ast S$ and $\bar{v}\in S\ast H$.
\end{fact}

For our purposes, it will suffice to look at the case in which $S=\set[1]$ is trivial and $u$ and $v$ are
elements of $F_n$. Regarding Fact \ref{separability and independence}, we see that independence of $u$ and $v$ over the empty set implies that the set $\set[u,v]$ is separable. So if $\set[u,v]$ is not separable, we know as well that $u$ and $v$ fork over the empty set.

\begin{defi}
Let $A$ be a set of words over $X$ representing elements in the free group $F_n=\langle e_1,\ldots ,e_n\rangle$. The \emph{Whitehead graph} of $A$, which we denote by $W_A$, is the graph with set of vertices $V(W_A)=\set[e_1,\ldots , e_n,e_1^{-1},\ldots , e_n^{-1}]$, and edges joining the vertices $u$ and $v^{-1}$ if and only if one of the words in $A$ cyclically contains the subword $uv$.
\end{defi}

\begin{defi}
Let $W$ be a graph. A vertex $u\in V(W)$ is called a \emph{cut vertex}, if removing $u$ and its adjacent edges leaves the graph disconnected.
\end{defi}

Whitehead graphs occur as projections of closed paths in certain 3-di\-men\-sio\-nal manifolds and were first introduced by Whitehead in \cite{WhiteheadOnCertainSets}. Using this topological picture
, Stallings showed the following fact which is crucial for our method to show that certain elements in free groups fork with each other.
 
\begin{fact}[{\cite[Theorem 2.4]{StalHandlebodies}}]
\label{separable -> cut vertex}
Let $A$ be a set of cyclically reduced words representing elements in $F_n$. If  $A$ is separable in $F_n$, the Whitehead graph $W_A$ has a cut vertex.
\end{fact}

Now we have all necessary tools to construct an independent sequence that witnesses the infinite weight of $p_0$. 

\begin{lem}
\label{my sequence}
The following sequence is independent over the empty set:
\begin{gather*}
(b_i)_{i< \omega}{\coloneqq}(e_2 e_1 e_2,\, e_3 e_2 e_1 e_2^2 e_3,\, \ldots ,\,e_{i+1} e_i \ldots e_2 e_1 e_2^2 e_3^2 \ldots e_i^2 e_{i+1},\, \ldots)\, ,
\end{gather*}
i.e. $b_0\coloneqq e_2 e_1 e_2$ and $b_{i}\coloneqq e_{i+2}b_{i-1}e_{i+1} e_{i+2}$ for $i\geq 1$.
\end{lem}
\begin{proof}
One can easily see that 
\begin{gather*}
\langle e_2 e_1 e_2,\, e_3 e_2 e_1 e_2^2 e_3,\,\ldots ,\,e_{n+1} e_n \ldots e_2 e_1 e_2^2 e_3^2 \ldots e_n^2 e_{n+1},\, e_{n+1}\rangle =F_{n+1} .
\end{gather*}
This suffices because as a special case of \ref{separability and independence}, we know that each basis of a free group forms an independent set.
\end{proof}
For the proof of Theorem \ref{Theoreminfiniteweight}, we will make use of the high level of connection in the Whitehead graphs of the elements $b_i$ (see Figure \ref{b_i Whitehead}). 
\begin{figure}
\begin{center}
\includegraphics{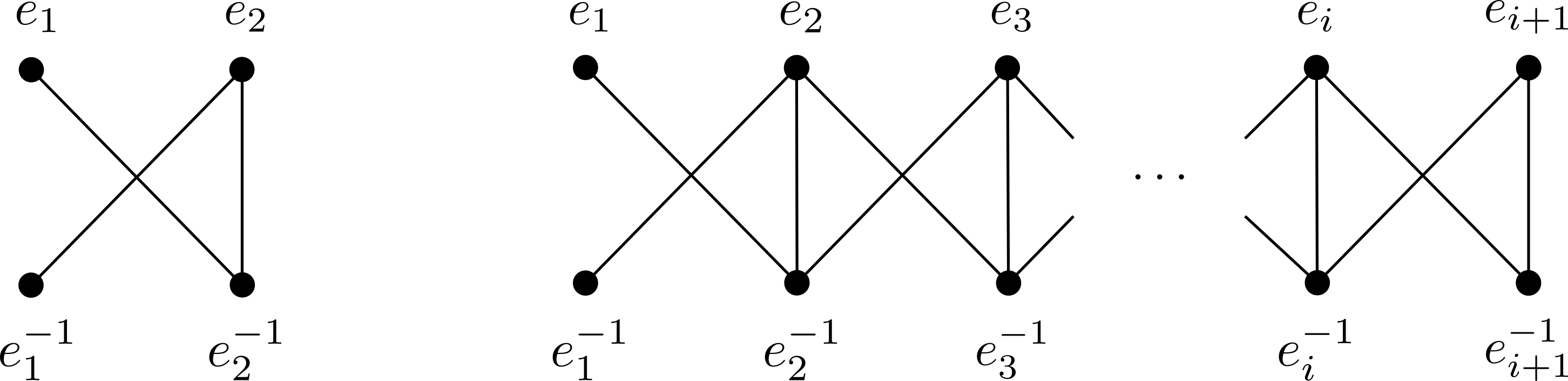}
\end{center}
\caption{The Whitehead graphs of $b_1=e_2 e_1 e_2$ and $b_i$. In both graphs, $e_2$ and $e_2^{-1}$ are the only cut vertices.}
\label{b_i Whitehead}
\end{figure}
\Theoreminfiniteweight*
\begin{proof}
Let $p(x)$ be a type over the empty set with a non-trivial realisation $a\in F\bsl\set[1]$ in some free group $F$. Fix a basis $X=\set[e_1,e_2,\ldots]$ of $F$. Permuting the elements of $X$ induces an automorphism of $F$ and thus does not change the type of $a$ over the empty set. So we may assume $a\in F_n$ for some $n\in \mathbb{N}$. 
As conjugating with an element of $F$ is also an automorphism, we can as well assume that $a$ is cyclically reduced.

Now take $(b_i)_{i<\omega}$ as defined in Lemma \ref{my sequence}. Using Whitehead graphs, we show that after leaving out the first elements of this sequence, the remaining sequence $(b_i)_{n\leq i<\omega}$ witnesses the infinite weight of $p$. By the last lemma, we already know that $(b_i)_i$ is an independent sequence. It remains to show that $a$ forks with $b_i$ over the empty set for all $i\geq n$. To do this, we show that the Whitehead graph $W_A$ of the set $A{\coloneqq}\set[a,b_i]$ has no cut vertex in the free group $F_{i+1}=\langle e_1,\ldots , e_{i+1}\rangle$. In this situation, we can apply Fact \ref{separable -> cut vertex} to see that there is no decomposition
\begin{equation*}
F_{i+1}=G\ast H
\end{equation*}
such that $a\in G$ and $b_i\in H$. This implies that $a$ forks with $b_i$ in $F_{i+1}$ and thus, as the embedding $F_{i+1}\hookrightarrow F$ is elementary (see Fact \ref{Answer Tarski problem}), they fork in $F$ as well and we are finished.

Permuting $X$ again, we may assume that $a$ contains the letter $e_1$. That means that in $W_A$, the vertices $e_1$ and $e_1^{-1}$ are each connected by an edge to at least one other vertex. If the only edges starting at $e_1^{\pm 1}$ end at $e_1^{\mp 1}$, we have $a=e_1^k$ for some $k\in \mathbb{Z}$. In this case, we can easily derive the infinite weight of $p$ from the infinite weight of $p_0$ because by \cite[Corollary 2.7]{PilForking}, we have  $p_0=tp(e_1/\emptyset)$.
So without loss, at least one of the vertices $e_1,e_1^{-1}$ is connected to a vertex $e_k^{\pm 1}$ with $1<k\leq n$. It follows from the definition of Whitehead graphs that in this case, they are in fact both connected to at least one other vertex. Applying another automorphism, we may assume that $a$ does neither contain $e_2$ nor its inverse such that $e_1$ is either connected to a vertex $e_k$ or $e_k^{-1}$ and that $e_1^{-1}$ is connected to $e_{k'}$ or $e_{k'}^{-1}$ where both $k$ and $k'$ are greater than 2 (we do not assume that those vertices are distinct). So $W(\set[a])$ contains at least the two edges shown in Figure \ref{edges in Whitehead}.

\begin{figure}
\begin{center}
\includegraphics{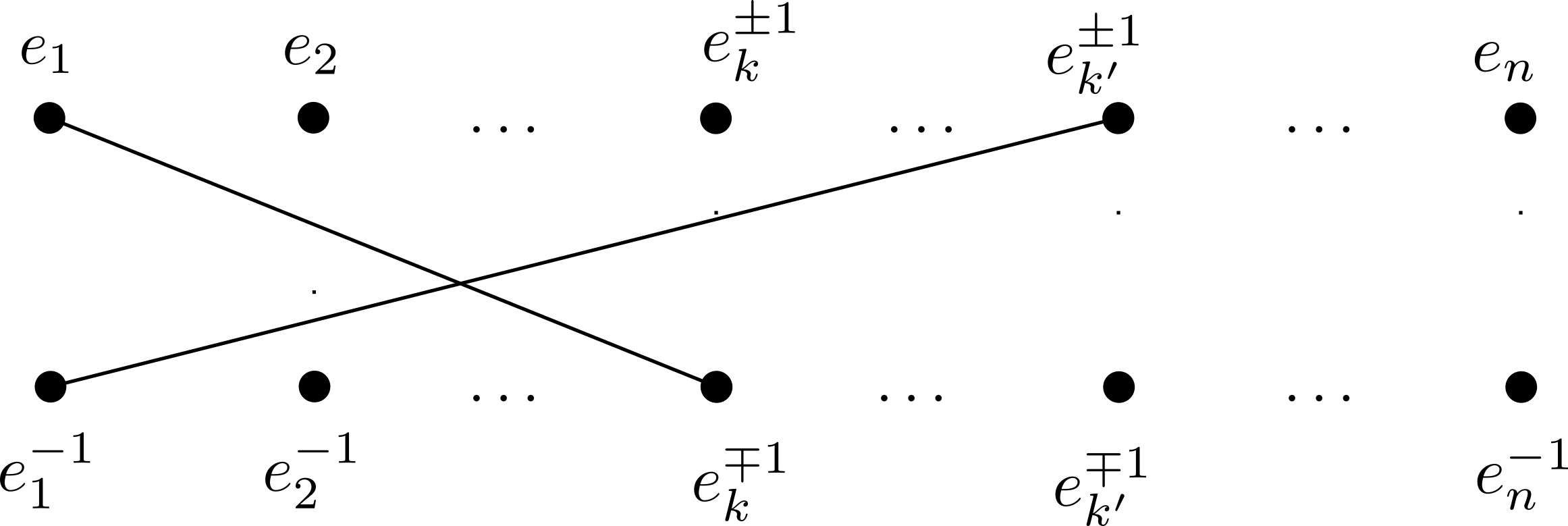}
\end{center}
\caption{Edges in $W(\set[a])$}
\label{edges in Whitehead}
\end{figure}

On the other hand, we already know  that the Whitehead  graph $W_{\set[b_i]}$ is of the form shown in Figure \ref{b_i Whitehead}.
Since $W_A$ is the union of $W_{\set[a]}$ and $W_{\set[b_i]}$, one sees that it has no cut vertex, because removing $e_2^{\pm 1}$ no longer disconnects $e_1^{\mp 1}$ from the rest of the graph. This finishes the proof.
\end{proof}

\bibliographystyle{amsalpha}
\bibliography{mybibliography}
 \bigskip
  \footnotesize
  Benjamin Br\"uck\\ \textsc{Fakult\"at f\"ur Mathematik \\
Universit\"at Bielefeld\\
Postfach 100131\\
D-33501 Bielefeld\\
Germany} \\ \nopagebreak
  \texttt{benjamin.brueck@uni-bielefeld.de}

\end{document}